\documentclass[11pt,a4paper]{article}
\usepackage{epsf,algorithm,epsfig,amsfonts,amsgen,amsmath,amstext,amsbsy,amsopn,amsthm,lineno}
\usepackage{color}
\usepackage{tikz}
\usepackage{cite}
\usepackage{enumerate}
\usepackage{float}
\usepackage{epstopdf}
\usepackage{subcaption}
\usepackage{multicol}
\usepackage{amssymb}
\usetikzlibrary{decorations.pathreplacing,calligraphy}
\usetikzlibrary{calc}

\newtheorem{theorem}{Theorem}

\newtheorem{lemma}{Lemma}

\newtheorem{problem}{Problem}

\newtheorem{claim}{Claim}

\newcounter{mathitem}

\newcommand{\chiS}{\chi_S}

\usetikzlibrary{decorations.markings}
\tikzstyle{vertex}=[circle, draw, inner sep=0pt, minimum size=5pt]

\begin{document}

\title{\bf\Large On the maximal anti-Ramsey problem of Burr, Erd\H{o}s, Graham, and S\'{o}s for $P_4$\thanks{This research was supported by National Key Research and Development Program of China (No.~2023YFA1010203), and National Natural Science Foundation of China (Nos. ~12401464 and 12471334).}
}
\date{}
\author{Zixuan Yang$^{a}$\thanks{Corresponding author.}
~\\[2mm]
\small $^{a}$School of Mathematics,  Northwest University, Xi'an, Shaanxi, P.R. China\\}
\maketitle

\begin{abstract}
Given a graph $L$, the maximal anti-Ramsey function $\chiS(n,e,L)$ denotes the minimum integer $\chiS$ for which there exists an $n$-vertex graph $G$ with at least $e$ edges admitting an edge-coloring with $\chiS$ colors in which each copy of $L$ in $G$ is rainbow.
In 1989, Burr, Erd\H{o}s, Graham, and S\'{o}s  posed the following problem:
Is it true that for all $\epsilon>0$, there exists $c(\epsilon)>0$ such that for all sufficiently large $n$,
$
\chiS\left(n,\binom{n}{2}-\lfloor n^{2-\epsilon}\rfloor,P_4\right)>c(\epsilon)n^2.
$
Very recently, Li, Ning, and Xie  gave a negative answer to the  problem for all $0< \epsilon< 1/2$.
In this note, we establish that a quadratic lower bound holds in the complementary regime $ \epsilon\geq 1/2$.
More specifically, we prove that for all $\epsilon\ge 1/2$ and sufficiently large $n$, there is an absolute constant $c>0$ such that
$
\chiS\left(n,\binom{n}{2}-\lfloor n^{2-\epsilon}\rfloor,P_4\right)>c n^2.
$

\medskip

\smallskip
\end{abstract}


\setcounter{footnote}{0}
\renewcommand{\thefootnote}{}
\footnotetext{E-mail addresses:  {\tt  yangzixuan@nwpu.edu.cn}}

\section{Introduction}

Throughout this paper all graphs are finite, undirected, and simple.   
Let $G$ be a graph with vertex set $V(G)$ and edge set $E(G)$.  
For a vertex $u$ of $G$, the \emph{degree} of $u$ in $G$ is denoted by $d_G(u)$.  
For $S\subseteq V(G)$, the subgraph of $G$ induced by $S$ is denoted by $G[S]$.
 Given  positive integer $n$,  we use  $K_n$, $C_n$, and $P_n$ to denote the complete graph, the cycle, and the path of $n$ vertices, respectively. 
The \textit{complement graph} (or simply \textit{complement}) of $G$, denoted $\overline{G}$, is the unique simple graph $\overline{G} = (V(G), \overline{E(G)})$ on the identical vertex set $V(G)$, whose edge set $\overline{E(G)}$ satisfies
$
\overline{E} = \bigl\{ \{u,v\} \subseteq V(G) \,\big|\, u \neq v,\ \{u,v\} \notin E(G) \bigr\}.
$

A \emph{matching} of a graph $G$ is a  subset of $E(G)$ such that no two edges share a vertex in common.
An edge set $M\subseteq E(G)$ is an \emph{induced matching} if the subgraph of
$G$ induced by the endpoints of the edges in $M$ has edge set exactly $M$.  In
particular, no two edges of $M$ are adjacent, and there is no edge of $G$ joining
endpoints of two different edges of $M$.

Ramsey theory asserts, in many different forms, that sufficiently large structures contain highly organized substructures. 
In the graph setting, the ambient structure is usually an edge-colored graph, and the organized subgraph is monochromatic. Anti-Ramsey theory, initiated by Erd\H{o}s, Simonovits, and S\'{o}s \cite{Erd3}, asks for the opposite extreme: a subgraph is called \textit{rainbow} if all its edges have distinct colors, and one seeks conditions forcing such rainbow subgraphs.

Burr, Erd\H{o}s, Graham, and S\'{o}s \cite{Bur2} introduced a dual maximal anti-Ramsey parameter, denoted $\chi_S(G,L)$.  
For any two graphs $G$ and $L$, let $\chi_S(G,L)$ be the minimum number of colors in an edge-coloring of $G$ in which every copy of $L$ is rainbow.  
For positive integers $n,e$, define
\[
\chiS(n,e,L)=\min\{\chiS(G,L): |V(G)|=n,\ |E(G)|\ge e\}.
\]
The function $\chi_S(n,e,L)$, referred to as the \emph{maximal anti-Ramsey function},  has an immediate connection to extremal graph theory. Here, one of the central objects of study is the Turán number of a graph. The $n$-vertex \textit{Turán number} of $L$, denoted $\mathrm{ex}(n,L)$, is defined as the maximum number of edges in an $n$-vertex graph that contains no copy of $H$. Assuming $L$ has at least two edges, it is immediate that $\chiS(n,e,L) = 1$ if and only if $e \le \mathrm{ex}(n,L)$. This implies that the asymptotic behavior of $\chi_S(n,e,L)$ is only non-trivial when $e > \mathrm{ex}(n,L)$. A natural foundational stability question addresses the threshold case $e = \mathrm{ex}(n,L)+1$. This regime was first investigated  by  Burr, Erd\H{o}s, Graham, and S\'{o}s \cite{Bur2}, who focused on $\chi_S\left(n, \lfloor n^2/4 \rfloor + 1, C_{2k+1}\right)$.
They also conjectured that, for all integers $k \ge 3$,
$
\chi_S\left(n, \lfloor n^2/4 \rfloor + 1, C_{2k+1}\right) = n^2/8 + o(n^2).
$
Recently, Buci\'{c}, Chen, and Ma \cite{Buc} proved that for an integer $k \ge 4$ and $n^2/4 + 1 \le e \le \binom{n}{2}$,
\[
\chi_S\left(n,e,C_{2k+1}\right) = \frac{e}{2} + \frac{n}{2}\sqrt{e - \frac{n^2}{4}} + o(n^2),
\]
confirming the conjecture for all $k \ge 4$ in a stronger form. Thus, the case $k = 3$ remains the only open case.

In addition to odd cycles, several results have been obtained when $L$ is a bipartite graph. Burr, Erd\H{o}s, Graham, and S\'{o}s \cite{Bur1} have investigated both cases regarding whether $L$ contains two strongly independent edges. Burr, Erdős, Graham, and Sós \cite{Bur2} determined the values of $\chi_S(n, \lfloor un \rfloor, P_k)$ for $k=3$, and for the range $u \ge k \ge 5$. Later, Sárközy and Selkow \cite{Sar} obtained a linear lower bound for connected bipartite graphs other than complete bipartite graphs.

In this paper, we focus on the case $L = P_4$, which is one of the central small bipartite cases in the work of Burr, Erd\H{o}s, Graham, and S\'{o}s \cite{Bur2}. It is closely related to Ruzsa–Szemerédi graphs, since an edge-coloring whose color classes are induced matchings automatically guarantees that every $P_4$ is rainbow. Fox, Huang, and Sudakov \cite{Fox} studied Ruzsa-Szemerédi graphs whose induced matchings have linear size.
Burr, Erd\H{o}s, Graham, and S\'{o}s \cite{Bur2} posed the following  open problem for $P_4$.

\begin{problem}[Burr, Erd\H{o}s, Graham, and S\'{o}s, \cite{Bur2}]
\label{prob:1.3}
Is it true that for all $\epsilon > 0$, there exists $c(\epsilon) > 0$ such that for all sufficiently large $n$, $\chi_S\left(n, \binom{n}{2} - \lfloor n^{2-\epsilon} \rfloor, P_4\right) > c(\epsilon)n^2$?
\end{problem}

Very recently, Li, Ning, and Xie \cite{Li} proved that the quadratic lower bound in Problem \ref{prob:1.3} fails for all  \(0<\epsilon <1/2\). Their proof relies on a construction originally due to Alon, Moitra, and Sudakov \cite{Alo}, which they adapt in \cite{Li} to demonstrate that the required quadratic bound cannot hold for any fixed \(\epsilon\) in the interval \((0,1/2)\).

In this note, we complement this problem by showing the quadratic lower bound is valid for all \(\epsilon \ge 1/2\).

\begin{theorem}\label{thm:main}
For every fixed $\epsilon\ge 1/2$, there exists $n_0=n_0(\epsilon)$ such that,
for every $n\ge n_0$,
\[
\chiS\left(n,\binom{n}{2}-\lfloor n^{2-\epsilon}\rfloor,P_4\right)>
\frac{n^2}{72}.
\]
\end{theorem}



\section{Proof of Theorem \ref{thm:main}}

Before proving the Theorem \ref{thm:main}, we first show the following lemma.

\begin{lemma}\label{lem:dense-core}
Let $G$ be an $n$-vertex graph such that its complement $H=\overline G$ satisfies $|E(H)|\le n^{3/2}$.  
Let $B=\{v\in V(G): d_H(v)>8\sqrt n\}$, $U=V(G)\setminus B$, and  $F=G[U].$
For all sufficiently large $n$, the following statements hold:
\begin{itemize}
\item [\rm{(i)}]  if the edges of $G$ are colored such that every copy of $P_4$ in $G$ is rainbow, then every color class restricted to $E(F)$ is an induced matching in $F$;
\item [\rm{(ii)}]  if $E(F)$ can be partitioned into $q$ induced matchings,  then $q> n^2/72$.
\end{itemize}
\end{lemma}

\begin{proof}

We first consider the size of $E(F)$.

\begin{claim}\label{c4}
$|E(F)|>\frac{n^2}{4}$.
\end{claim}

\begin{proof}[Proof of Claim \ref{c4}]
Since $|E(H)|\le n^{3/2}$, the degree sum in $H$ satisfies
\[
\sum_{v\in V(G)}d_H(v)=2|E(H)|\le 2n^{3/2}.
\]
Then by the definition of $B$, we have
\[
8\sqrt n\,|B|<\sum_{v\in B}d_H(v)\le \sum_{v\in V(G)}d_H(v)\le 2n^{3/2},
\]
and therefore $|B|<n/4$, which implies that $|U|>3n/4$. 

Notice that the complement of $F$  is $H[U]$.  So we have
\begin{align}\label{eq1}
|E(F)|=\binom{|U|}{2}-|E(H[U])|
       \ge \binom{|U|}{2}-|E(H)|.
\end{align}
Substitute the inequalities $|U|>3n/4$ and $|E(H)|\le n^{3/2}$ into inequality (\ref{eq1}) to obtain the lower bound
\[
|E(F)|>\binom{3n/4}{2}-n^{3/2}
       =\frac{9n^2}{32}-\frac{3n}{8}-n^{3/2}.
\]
Since $n^2/32-n^{3/2}-3n/8\ge 0$ for  sufficiently large $n$,
we get that 
\[
|E(F)|> \frac{8n^2}{32}=\frac{n^2}{4}.
\]
Thus, the claim holds.
\end{proof}

We now prove the statement (i) by two claims. 

\begin{claim}\label{c1}
Any two adjacent edges in $F$ cannot have the same color. 
\end{claim}

\begin{proof}[Proof of Claim \ref{c1}]
By contradiction. 
Suppose that there exist two edges $xy,xz\in E(F)$ such that the colors of $xy$ and $xz$ are the same.  
Since $y\in U$, we have
$d_H(y)\le 8\sqrt n$, and consequently
\[
d_G(y)=n-1-d_H(y)\ge n-1-8\sqrt n>2,
\]
for large enough $n$. 
Thus $y$ has a neighbor $w$ in $G$ with $w\notin\{x,z\}$.  
One can see that the three edges $wy,yx,xz$ form a copy of $P_4$ in $G$, where the four vertices $w,y,x,z$ are distinct.  
Therefore, we obtain a path $P_4$ which  is not rainbow since $yx$ and $xz$ have the same color, a contradiction.  
\end{proof}

\begin{claim}\label{c2}
For any two disjoint edges  $e_1, e_2$ in $F$ with the same color, there exist no edges of $G$ connecting an endpoint of $e_1$ to an endpoint of $e_2$. 
\end{claim}

\begin{proof}[Proof of Claim \ref{c2}]

By contradiction. Suppose that there exist two disjoint edges $ab,cd\in E(F)$  with the same color such that  
 an edge of $G$ joins one endpoint of $ab$ to one endpoint of $cd$.
Then after relabeling the endpoints, without loss of generality, we may assume that $bc\in E(G)$.  
This implies that three edges $ab,bc,cd$  form a $P_4=abcd$ in $G$, and one can see that this copy of $P_4$ has two  edges with the same colors (namely $ab$ and $cd$), which  contradicts the fact that every copy of $P_4$ in $G$ is rainbow.  Therefore, the claim holds.
\end{proof}

By Claim \ref{c1},  each color class on $F$ is a matching, and then by Claim \ref{c2}, each color class
restricted to $F$ is  an induced matching in $F$. Therefore, the statement (i) holds.

In the following, we prove the statement (ii).
Let $M_1,M_2,\ldots,M_q$ be the nonempty induced matchings in $F$.
Write $ m=\sum_{i=1}^q |M_i|$. Then $m=|E(F)|$.
Let 
\[
\mathcal T=\{(v,i,e): v\in U,\ e\in M_i,\ v\text{ is covered by }M_i,
\text{ and }v\text{ is not an endpoint of }e\}.
\] 

We next count  the size of $\mathcal T$ by using double-counting.
For any fixed $i\in [q]$, since $M_i$ is a matching with $|M_i|$ edges, it covers exactly $2|M_i|$
vertices.  For each covered vertex $v$, there is exactly one edge of $M_i$
incident with $v$, and hence there are $|M_i|-1$ choices for an edge
$e\in M_i$ not incident with $v$.  Therefore
\begin{equation}\label{eq:T-lower}
|\mathcal T|=\sum_{i=1}^q 2|M_i|(|M_i|-1).
\end{equation}

\begin{claim}\label{c3}
For any fixed  vertex $v\in U$,  the number of triples in
$\mathcal T$ with first coordinate $v$ is at most $\binom{d_{\overline {F}}(v)}{2}$.  
\end{claim}

\begin{proof}[Proof of Claim \ref{c3}]
Let $(v,i,e)\in\mathcal T$, and write $e=xy$.  
Since $v$ is covered by $M_i$, there is a unique edge $f\in M_i$ incident with $v$, where $1\le i\le q$.  
Note that the edge $e$ is not incident with $v$. 
So $e$ and $f$ are two distinct edges of the induced matching $M_i$, where $1\le i\le q$.  
One can see that both $x$ and $y$ belong to $N_{\overline F}(v)$.
Indeed, otherwise, if $v$ were adjacent in $F$ to $x$ or to $y$, then there exists a edge of $F$  joining an
endpoint of $f$ to an endpoint of $e$, which contradicts the fact that $M_i$ is induced.

We now show that the map $(v,i,e)\mapsto \{x,y\}$ is injective.  
Indeed, the edge $e$ belongs to exactly one member of the partition, so the pair $\{x,y\}$ determines both $e$ and its index $i$.  
Hence the number of triples with first coordinate $v$ is at most the number of two-element subsets of
$N_{\overline F}(v)$, namely $\binom{d_{\overline {F}}(v)}{2}$. This completes the proof of Claim \ref{c3}.  
\end{proof}

Summing over $v\in U$ and by Claim \ref{c3} and \eqref{eq:T-lower}, we obtain that 
\begin{equation}\label{eq:key-count}
\sum_{i=1}^q 2|M_i|(|M_i|-1)=|\mathcal T|
\le \sum_{v\in U}\binom{d_{\overline {F}}(v)}{2}.
\end{equation}

Now we establish an upper bound of $\sum_{v\in U}\binom{d_{\overline {F}}(v)}{2}$.
By the definition of $U$, we have $d_H(v)\le 8\sqrt n$ for any $v\in U$.
 Since $\overline {F}=H[U]$, we have 
 \begin{align}\label{eq0}
 d_{\overline {F}}(v)=d_{H[U]}(v)\le d_H(v)\le 8\sqrt n.
 \end{align}
and consequently,
\begin{align}\label{eq00}
\sum_{v\in U}d_{\overline {F}}(v)=2|E(H[U])|
\le 2|E(H)|\le 2n^{3/2}.
\end{align}
By (\ref{eq0}) and (\ref{eq00}), we have
\begin{align}\label{eq5}
\sum_{v\in U}\binom{d_{\overline {F}}(v)}{2}&=\sum_{v\in U}\frac{d_{\overline {F}}(v)(d_{\overline {F}}(v)-1)}{2}\notag\\
&\le \sum_{v\in U} \frac{d_{\overline {F}}(v)\cdot 8\sqrt n}{2}\notag\\
&= 4\sqrt n\sum_{v\in U}d_{\overline {F}}(v)\notag\\
&\le 4\sqrt n\cdot 2n^{3/2}=8n^2.
\end{align}

Combining \eqref{eq:key-count} and (\ref{eq5}), we have
\begin{equation}\label{eq:ki-ki-1}
\sum_{i=1}^q (|M_i|^2-|M_i|)=\sum_{i=1}^q |M_i|(|M_i|-1)\le 4n^2.
\end{equation}
Note that 
\[
\sum_{i=1}^q |M_i|=m=|E(F)|\le \binom n2<n^2/2.
\]
So by  \eqref{eq:ki-ki-1}, we get
\begin{equation}\label{eq:square-sum}
\sum_{i=1}^q |M_i|^2=m+\sum_{i=1}^q  |M_i|( |M_i|-1)<\frac{n^2}{2}+4n^2=\frac{9n^2}{2}.
\end{equation}

Notice that
\[
m^2=\left(\sum_{i=1}^q |M_i|\right)^2\le q\sum_{i=1}^q |M_i|^2,
\]
by the Cauchy-Schwarz inequality, and $m=|E(F)|> n^2/4$ by Claim \ref{c4}. Thus, by \eqref{eq:square-sum}, we have
\[
q\ge \frac{m^2}{\sum_{i=1}^q |M_i|^2}
>\frac{(n^2/4)^2}{9n^2/2}
=\frac{n^2}{72}.
\]
This completes the proof of Lemma \ref{lem:dense-core}.
\end{proof}


\begin{proof}[Proof of Theorem~\ref{thm:main}]

Fix an arbitrary constant $\epsilon \ge 1/2$. Let $G$ be any $n$-vertex graph with
\[
|E(G)| \ge \binom{n}{2} - \left\lfloor n^{2-\epsilon} \right\rfloor.
\]
Then the complement graph $\overline{G}$ contains at most $\left\lfloor n^{2-\epsilon} \right\rfloor$ edges. Since $\epsilon \ge 1/2$, we have $2-\epsilon \le 3/2$, which implies that
\[
|E(\overline{G})| \le n^{2-\epsilon} \le n^{3/2}.
\]

Now we consider any edge-coloring of $G$ such that every copy of $P_4$ in $G$ is rainbow. We apply Lemma~\ref{lem:dense-core} to $G$, which yields a vertex subset $U \subseteq V(G)$ and the induced subgraph $F = G[U]$. By Lemma~\ref{lem:dense-core}  (i), when restricted to $E(F)$, the nonempty color classes of the coloring form a partition of $E(F)$ into induced matchings of $F$. Furthermore, by Lemma~\ref{lem:dense-core}  (ii), the number of such induced matchings (i.e., the number of colors appearing on $E(F)$) is strictly greater than $n^2/72$ for all sufficiently large $n$.

Since the original edge-coloring of $G$ uses at least as many colors as those appearing on $E(F)$, every rainbow $P_4$-coloring of $G$ requires more than $n^2/72$ colors.
As the above bound holds for every $n$-vertex graph $G$ with at least $\binom{n}{2} - \left\lfloor n^{2-\epsilon} \right\rfloor$ edges, so taking the minimum over all such graphs yields
\[
\chiS\left(n, \binom{n}{2} - \left\lfloor n^{2-\epsilon} \right\rfloor, P_4\right) > \frac{n^2}{72}
\]
for all sufficiently large $n$, which completes the proof of Theorem \ref{thm:main}.
\end{proof}





\begin{thebibliography}{30}


\bibitem{Alo} N. Alon, A. Moitra, and B. Sudakov, Nearly complete graphs decomposable into large induced matchings and their applications, \emph{J. Eur. Math. Soc.}, \textbf{15} (2013), 1075--1096.

\bibitem{Buc} M. Buci\'{c}, K. Chen, and J. Ma, On a maximal anti-Ramsey conjecture of Burr, Erd\H{o}s, Graham, and S\'{o}s, arXiv:2603.18952, 2026.


\bibitem{Bur1} S. A. Burr, P. Erd\H{o}s, P. Frankl, R. L. Graham, and V. T. S\'{o}s, Further results on maximal anti-ramsey graphs, in \emph{Graph Theory, Combinatorics, and Applications, Vol. I}, Y. Alavi, A. Schwenk (Eds.), John Wiley and Sons, New York, 1988, 193--206.

\bibitem{Bur2} S. A. Burr, P. Erd\H{o}s, R. L. Graham, and V. T. S\'{o}s, Maximal antiramsey graphs and the strong chromatic number, \emph{J. Graph Theory}, \textbf{13} (1989), no. 3, 263--282.





\bibitem{Erd3} P. Erd\H{o}s, M. Simonovits, and V. T. S\'{o}s, Anti-Ramsey theorems, in \emph{Infinite and finite sets (Colloq., Keszthely, 1973); dedicated to P. Erdős on his 60th birthday}, Vol. II, Colloq. Math. Soc. J\'{a}nos Bolyai, Vol. 10, North Holland, Amsterdam, 633--643.





\bibitem{Fox} J. Fox, H. Huang, and B. Sudakov, On graphs decomposable into induced matchings of linear sizes, \emph{Bull. Lond. Math. Soc.}, \textbf{49} (2017), no. 1, 45--57.



\bibitem{Li} M. Li, B. Ning, and T. Xie, Two problems of Burr, Erd\H{o}s, Graham, and S\'{o}s on maximal anti-Ramsey functions for $P_4$, arXiv:2606.30505v1, 2026.


\bibitem{Sar} G. N. Sárközy and S. M. Selkow, On an anti-Ramsey problem of Burr, Erd\H{o}s, Graham, and T. S\'{o}s, \emph{J. Graph Theory}, \textbf{52} (2006), no. 2, 147--156.



\end{thebibliography}
\end{document}